\documentclass[12pt,reqno]{article}

\usepackage[usenames]{color}
\usepackage{amssymb}
\usepackage{amsmath}
\usepackage{amsthm}
\usepackage{amsfonts}
\usepackage{amscd}
\usepackage{graphicx}
\usepackage{slashbox}
\usepackage[colorlinks=true,
linkcolor=webgreen,
filecolor=webbrown,
citecolor=webgreen]{hyperref}

\definecolor{webgreen}{rgb}{0,.5,0}
\definecolor{webbrown}{rgb}{.6,0,0}

\usepackage{color}
\usepackage{fullpage}
\usepackage{float}
\usepackage{graphics}
\usepackage{latexsym}
\usepackage{epsf}
\usepackage{breakurl}

\newcommand{\seqnum}[1]{\href{https://oeis.org/#1}{\underline{#1}}}

\usepackage{enumerate,tikz,listings,array,longtable,url,hyperref}
\usetikzlibrary{shapes,arrows,chains}
\usepackage{pgf}

\begin{document}

\theoremstyle{plain}
\newtheorem{theorem}{Theorem}
\newtheorem{corollary}[theorem]{Corollary}
\newtheorem{lemma}[theorem]{Lemma}
\newtheorem{proposition}[theorem]{Proposition}

\theoremstyle{definition}
\newtheorem{definition}[theorem]{Definition}
\newtheorem{example}[theorem]{Example}
\newtheorem{conjecture}[theorem]{Conjecture}
\newtheorem{problem}[theorem]{Problem}

\theoremstyle{remark}
\newtheorem{remark}[theorem]{Remark}
\newtheorem{question}[theorem]{Question}

\begin{center}
\vskip 1cm{\LARGE\bf 
Run Distribution Over Flattened Partitions}
\vskip 1cm
\large
Olivia Nabawanda\footnote{Corresponding author.}\\
Department of Mathematics\\Makerere University\\
P.O.Box 7062\\
Kampala, Uganda\\
\href{mailto:onabawanda@must.ac.ug}{\tt onabawanda@must.ac.ug}\\
\ \\
Fanja Rakotondrajao\\
D\'epartement de  Math\'ematiques et Informatique\\
Universit\'e d'Antananarivo, Antananarivo\\
Antananarivo,  Madagascar\\
\href{mailto:frakoton@yahoo.fr}{\tt frakoton@yahoo.fr}\\
\ \\
Alex Samuel Bamunoba\\
Department of Mathematics\\Makerere University\\
P.O.Box 7062\\
Kampala, Uganda\\
\href{mailto:bamunoba@cns.mak.ac.ug}{\tt bamunoba@cns.mak.ac.ug}
\end{center}

\vskip .2 in

\begin{abstract}
The study of flattened partitions is an active area of current research. In this paper, our study unexpectedly leads us to the OEIS numbers \seqnum{A124324}. We provide a new combinatorial interpretation of these numbers. A combinatorial bijection between flattened partitions over $[n+1]$ and the partitions of $[n]$ is also given in a separate section. We introduce the numbers $f_{n, k}$ which count the number of flattened partitions over $[n]$ having $k$ runs. We give recurrence relations defining them, as well as their exponential generating function in differential form. It should be appreciated if its closed form is established. We extend the results to flattened partitions where the first $s$ integers belong to different runs. Combinatorial proofs are given.
\end{abstract}
\newcommand{\C}{\mathbb{C}}
\newcommand{\PP}{\mathrm{P}}
\newcommand{\B}{\mathbb{B}}
\newcommand{\Z}{\mathbb{Z}}
\newcommand{\N}{\mathbb{N}}
\newcommand{\R}{\mathbb{R}}
\newcommand{\Q}{\mathbb{Q}}
\newcommand{\fS}{\mathfrak{S}}
\newcommand{\Id}{\mathrm{Id}}
\newcommand{\cone}{\mathrm{cone}}
\newcommand{\conv}{\mathrm{conv}}
\newcommand{\height}{\mathrm{height}}
\newcommand{\im}{\mathrm{im}}
\newcommand{\HH}{\mathcal{H}}
\newcommand{\fpa}{\mathrm{FPA}}
\newcommand{\pn}{\mathrm{part}}
\newcommand{\RP}{\mathrm{RelPrime}}
\newcommand{\rp}{\mathrm{relprime}}
\newcommand{\rpac}{\mathrm{rpac}}
\newcommand{\relprime}{\mathrm{relprime}}
\newcommand{\ptn}{\mathrm{part}}
\newcommand{\Rpac}{\mathrm{Rpac}}
\newcommand{\Relprime}{\mathrm{Relprime}}
\newcommand{\Ptn}{\mathrm{Part}}
\def\v{{\boldsymbol v}}
\def\t{{\boldsymbol t}}
\newcommand{\A}{\mathcal{V}}
\newcommand{\m}{\mathfrak{m}}

\section{Introduction and preliminaries}\label{sec:intro}
The study of the different statistics of permutations such as descents, ascents, excedances and runs has a long history and has been an area of intensive research in the past years. A lot of work is available in the literature \cite{ehrenborg2000excedance, macmahon2001combinatory, mantaci2003exceedingly, sloane2008line, con1}. These statistics play an important role in combinatorics and mathematics as a whole.

For a given positive integer $n$, we will denote the set $\{1, 2, \ldots, n\}$  by $[n]$. A permutation $\sigma$ over $[n]$ will be represented as a word $\sigma(1)\sigma(2)\cdots\sigma(n)$. We say that $\sigma$ has an \textit{ascent (descent)} at position $i$ if $\sigma(i) < \sigma(i+1)\,(\sigma(i) > \sigma(i+1))$, where $i \in [n]$. A \textit{run} in a permutation $\sigma$ is a subword $\sigma(i)\sigma(i+1)\cdots \sigma(i+p)\sigma(i+p+1)$ where $i, i+1, \ldots, i+p$ are consecutive ascents. In this case, $i-1$ (if it does exist) and $i+p+1$ are non ascents, where $i \in [n]$. For example, in $\sigma = 526134$, we have ascents at positions $2, 4$ and $5$. It also has descents at positions $1$ and $3$. The last element $4$ of $\sigma$ at position $6$ is neither an ascent nor a descent.  A \textit{right to left minimum} of a permutation $\sigma$ is an element $\sigma(i)$ such that $\sigma(i) < \sigma(j)$ for all $j > i$. For example, in $\sigma = 1246357$, the\textit{ right to left minima} are $\{1, 2, 3, 5, 7\}$.

Counting permutations according to the number of runs has been studied from various perspectives in enumerative combinatorics. Canfield and Wilf \cite{canfield2008counting} considered a run as a subsequence of a permutation $\sigma$, whose values either increase on the interval (run up) or decrease on the interval (run down). More related work on permutation runs can also be found in articles \cite{bona2000combinatorial, ma2012explicit, ma2013enumeration, zhuang2016counting}. A permutation $\pi$ is said to be a \textit{flattened partition} if it consists of runs arranged from left to right such that their first entries are in increasing order. It is clear that the first run always starts with $1$, and so all flattened partitions start with the integer $1$. For example, consider a permutation $\sigma = 139278456$. This is a flattened partition with three runs namely: $139, 278, 456$ whose first entries $1, 2, 4$ are in increasing order. However, the permutation $\sigma = 139456287$ is not a flattened partition since the first entries $1, 4, 2, 7$, of the runs $139, 456, 28, 7$ are not in increasing order. 
Given a non-empty finite subset $S$ of positive integers, a \textit{set partition} $P$ of $S$ is a collection of disjoint non-empty subsets $B_{1},B_{2}, \ldots, B_{k}$ of $S$ (called blocks) such that $\displaystyle\cup_{i = 1}^{k}B_{i} = S$ \cite{mansour2012combinatoric, rota1964number}. We shall maintain the name and notion of ``flattened partition" introduced by Callan \cite{callan2009pattern}. Callan borrowed the notion ``\textit{flatten}" from \textit{$Mathematica^{\textregistered}$} programming language, where it acts by taking lists of sets arranged in increasing order, removes their parentheses, and writes them as a single list \cite{stephen1999mathematica}. Mansour et al.\ \cite{mansour2015counting} also used the same notion. To generate flattened partitions, the elements of each block are written as increasing subsequences, and blocks arranged from left to right in increasing order of their first entries.  Carlitz \cite{carlitz1968generalized} also applied the notion ``\textit{flatten}" to permutations expressed in cycle notation.

We will let $\mathcal{F}_{n}$ denote the set of all flattened partitions over $[n]$, $\mathcal{F}_{n, k}$ the set consisting of all flattened partitions over $[n]$ having $k$ runs, $f_{n, k}$ the cardinality of the set $\mathcal{F}_{n, k}$. In Table \ref{tab:t1}, we give the first few values of the numbers $f_{n, k}$.
\begin{table}[H]
\centering
\footnotesize
\begin{tabular}{|c| c c c c c|}
\hline
\backslashbox{$n$}{$k$} & $1$ & 2 & 3 & 4 & 5 \\
\hline
$1$ & 1 &  &  &   &   \\
2 & 1 &  &  &   &   \\
3 & 1 & 1 &  &   &   \\
4 & 1 & 4 & 0 &   &   \\
5 & 1 & 11 & 3 &  0 &   \\
6 & 1 & 26 & 25 & 0 & 0  \\
\hline
\end{tabular}
\caption{The numbers $f_{n, k}$}
\label{tab:t1}
\end{table}
We notice that the terms in the column for $k = 2$ of Table \ref{tab:t1} correspond to the Eulerian numbers. Foata and Sch\"{u}tzenberger \cite{olivia} gave the fundamental work on these numbers. Mantaci and Rakotondrajao \cite{mantaci} gave a new combinatorial interpretation to the same. Many other references concerning Eulerian numbers can be found on the OEIS \seqnum{A000295}.
In Section \ref{B}, we establish different recurrence relations of the numbers $f_{n, k}$ and give their combinatorial proofs. We also define the exponential generating function $F(x, u)$ of the numbers $f_{n, k}$ which is defined by \begin{equation*}
F(x, u) = \sum_{n \geq 0}\sum_{k\geq 0} f_{n, k} x^{k}\frac{u^{n}}{n!} = \sum_{n \geq 0}\sum_{\sigma \in \mathcal{F}_{n}}x^{run(\sigma)}\frac{u^{n}}{n!},
\end{equation*} where $run(\sigma)$ is the number of runs in a flattened partition $\sigma$. In Section \ref{C}, we generalize the results in Section \ref{B} to flattened partitions over $[n]$ whose first $s$ integers belong to different runs . We let $f_{n, k}^{(s)}$ denote the number of flattened partitions over $[n]$ whose first $s$ integers belong to different runs, and $F^{[s]}(x, u)$ the exponential generating function for the numbers $f_{n, k}^{(s)}$. The first few values of $f_{n, k}^{(s)}$ for $s = 2$ and $s = 3$ are shown in Table \ref{tab:t3} and \ref{tab:t4}:
\begin{table}[H]
\begin{minipage}[b]{0.45\linewidth}
\centering
\footnotesize
\begin{tabular}{|c| c c c c |}
\hline
\backslashbox{$n$}{$k$} & $2$ & 3 & 4 & 5 \\
\hline
$2$ & 0 &  &  &   \\
3 & 1 & 0 &  &   \\
4 & 3 & 0 &  &   \\
5 & 7 & 3 & 0 &   \\
6 & 15 & 22 & 0 &   \\
7 & 31 & 106 & 14 & 0 \\
\hline
\end{tabular}
\caption{The numbers $f_{n, k}^{(2)}$}
\label{tab:t3}
\end{minipage}
\begin{minipage}[b]{0.45\linewidth}
\centering
\footnotesize
\begin{tabular}{|c| c c c |}
\hline
\backslashbox{$n$}{$k$} & $3$ & 4 & 5 \\
\hline
4 & 0 &  &     \\
5 & 2 & 0 &    \\
6 & 12 & 0 &     \\
7 & 50 & 12 & 0    \\
8 & 180 & 139 & 0   \\
\hline
\end{tabular}
\caption{The numbers $f_{n, k}^{(3)}$}
\label{tab:t4}
\end{minipage}
\end{table}
Mansour et al.\ \cite{mansour2015counting} give a recursive formula for the number of flattened partitions over $[n]$, and also mention that the number of distinct permutations that can be obtained as flattened partitions over $[n]$ is the Bell
number $B_{n-1}$. Let $\mathcal{P}_{n}$ denote the collection of partitions of $[n]$. Our study led us to the OEIS \seqnum{A124324} which counts the number of partitions of $[n]$ having $k$ blocks of size greater than $1$. 

The authors of the present paper were introduced to the OEIS \seqnum{A124324} by Heinz \cite{fhh} who gave a maple program for computing the terms of this sequence recursively. The first maple program which computes this sequence using the exponential generating function was given by Emeric Deutsch. A combinatorial bijection between elements of $\mathcal{P}_{n}$ having $(k-1)$ blocks of size greater than $1$ and $\mathcal{F}_{n+1, k}$ will be given in Section \ref{A}.
\section{Flattened partitions and their behaviours on runs}\label{B} 
\subsection{Recurrence relations}
We have $f_{n, 1} = 1$ for all $n \geq 1$ and $f_{n, k} = 0$ for all $k\geq n\geq 2$. It is not possible to have a flattened partition over $[n]$ whose number of runs $k$ is greater or equal to its length.
\begin{theorem} \label{thm4}
For all integers $n$ and $k$ such that $2 \leq k < n$, the numbers $f_{n, k}$ of flattened partitions over $[n]$ with $k$ runs satisfy the recurrence relation \begin{equation}\label{mum1}
f_{n, k} = \sum_{m = 1}^{n-2}\bigg({n-1 \choose m} - 1\bigg)f_{m, k-1}.
\end{equation}
\end{theorem}
\begin{proof}  To construct a flattened partition $\pi$ over $[n]$ having $k$ runs, we consider a flattened partition $\tau$ over $[m]$ having $k-1$ runs, for an integer $m < n$. Since all flattened partitions start with element $1$, we insert the word of length $n-m$ starting with $1$, before a flattened partition $\tau$ to obtain $1\underbrace{\dots}_{n-m-1}\underbrace{\tau}_{m}$, and then re-order the elements of $\tau$ making sure the number of runs increase by one. The  $n-m-1$ elements between $1$ and $\tau$ can be chosen from the set $\{2, 3, 4, \ldots, n\}$  of $(n-1)$ terms in  $\displaystyle n-1 \choose n-m-1$ ways. In order to increase the number of runs by one, we avoid choosing the subset $\{2, 3, \ldots, n-m-1\}$ of consecutive elements.
We thus have $\displaystyle \bigg({n-1 \choose n-m-1} - 1\bigg)$ possible subsets to be inserted after $1$.  
The minimum number of elements in the first run is $2$, implying that the maximum length of $\tau$ is $n-2$. Since $n \geq 3$, then the minimum length of $\tau$ is $1$. Thus we have that $1 \leq m \leq n-2$. Thus summing up over $1 \leq m \leq n-2$ gives the recurrence relation in Equation \eqref{mum1}.
\end{proof} 
\begin{example}
Let us construct flattened partitions over $[6]$ having $3$ runs from a flattened partition $\tau$ over $[3]$ having $2$ runs. We have $\tau = 132$. 
The favorable subsets of two terms from the set $\{2, 3, 4, 5, 6\}$ are: $\{2, 4\}, \{2, 5\}, \{2, 6\}, \{3, 4\}, \{3, 5\}, \{3, 6\}, \{4, 5\}, \{4, 6\}, \{5, 6\}$.
Consider the pair $(\{3, 4\}, \tau)$, we get $\pi = 134265$.
For the pair $(\{4, 6\}, \tau)$, we get $\pi = 146253$.
\end{example}
Let $a_{n}$ denote the maximum number of runs $k$ in a flattened partition of $[n]$. From Table \ref{tab:t1} above, we see that the maximum number of runs, $a_{n}$ results into a sequence $1, 1, 2, 2, 3, 3, 4, 4, 5, 5, \ldots,$ for $n \geq 1$.
\begin{proposition}\label{pp}
The maximal number $a_{n}$ of runs in a flattened partition over $[n]$ satisfies the relation \begin{equation*}\label{ks}
a_{n} = a_{n-2} +1,
\end{equation*} for all $n\geq 2$, with initial conditions $a_{0} = 0, a_{1} = 1$.
\end{proposition}
\begin{proof}
Let $n$ be an integer such that $n \geq 2$. Consider a flattened partition $\tau$ over $[n-2]$ having maximal number of runs. Using the construction in Theorem \ref{thm4}, inserting two elements $1x$ where $x = \{3, 4, \ldots, n\}$ before $\tau$ and re-ordering the elements of $\tau$ can only add a maximum of $1$ run. Hence we have $a_{n} \geq a_{n-2} + 1$. On the other hand, inserting $n-2$ elements between $1$ and the identity $\tau = 1$ from the construction in the same theorem, we have $1\underbrace{\cdots}_{n-2}1$ and then re-ordering the elements of $\tau$. This means the subsets between $1$ and $\tau$ can be chosen from the set $\{3, 4, \ldots, n\}$ of $n-2$ elements whose maximum number of runs is $a_{n-2}$. Hence we have $a_{n} \leq a_{n-2}+1$. These two inequalities, together with $a_{0} = 0$ and $a_{1} = 1$, yield $a_{n} = a_{n-2}+1$.
\end{proof}
As an obvious conclusion, we have the following corollary:
\begin{corollary}
The maximal number of runs $a_{n}$ in a flattened partition over $[n]$ has the closed form $\displaystyle \dfrac{1}{4}\big(2n + (-1)^{n+1} +1\big)$ and its generating function $\displaystyle Y(x) = \sum_{n=0}^{\infty}a_{n}x^{n}$ is given by $\dfrac{1}{(1-x)^{2}(1+x)}$.
\end{corollary}

If $a$ is a starting point of a run in a flattened partition $\pi$ and $x$ is an integer such that $x < a$, then $\pi^{-1}(x) < \pi^{-1}(a)$. In other words, all integers smaller than $a$ are on its left. By contradiction, suppose there exists an $x < a$ on the right of $a$. Then $x$ is an element of another run. This makes the starting points of the runs of $\pi$ not to be in increasing order, which contradicts $\pi$ being a flattened partition. Hence $x$ should be on the left of $a$.

Let $\mathcal{C}_{n, k}$ denote the set of flattened partitions over $[n]$ of the form \textit{$1X2\cdots $}, having $k$ runs, where $X \in \{3, 4, 5, \ldots, n\}$. More precisely, each partition in $\mathcal{C}_{n, k}$ has only two elements in the first run. For example, we have $\mathcal{C}_{5, 2} = \{13245, 14235, 15234\}$.

From the construction in Theorem \ref{thm4}, we have the following corollary:
\begin{corollary}\label{sss}
For all integers $n$ and $k$ such that $1 \leq k < n$, the cardinality of the set $\mathcal{C}_{n, k}$ is $(n-2)f_{n-2, k-1}$.
\end{corollary}
Let $\mathcal{K}_{n, k}$ denote the set of flattened partitions over $[n]$ having $k$ runs and containing either the subword $\mathbf{2n1}$ or  the integer $\mathbf{n}$ at the end. In other words, deleting the integer $\textbf{n}$ does not affect the number of runs. Let $\mathcal{L}_{n, k}$ denote the set of flattened partitions over $[n]$ having $k$ runs and containing the subword $\mathbf{1n2}$, the set where deleting the integer $\textbf{n}$ reduces the number of runs by $1$. For example, we have $\mathcal{L}_{5, 2} = \{12354, 12534, 15234\}$ and $\mathcal{K}_{5, 2} = \{14523, 14235, 13524, 13452, 13425, 13245, 12453, 12435\}$.
\begin{remark}\label{rmk}
The sets $\mathcal{K}_{n, k}$ and $\mathcal{L}_{n, k}$ are mutually exclusive, i.e., $\mathcal{K}_{n, k} \cap \mathcal{L}_{n, k} = \emptyset$ and form a partition of the set $\mathcal{F}_{n, k}$, i.e., $\mathcal{F}_{n, k} = \mathcal{K}_{n, k} \cup \mathcal{L}_{n, k}$.
\end{remark}

For all integers $n$ and $k$ such that $1 \leq k <n$, the cardinalities of the sets $\mathcal{C}_{n, k}$ and $\mathcal{L}_{n, k}$ are equal. We will use a variation of the bijection defined by Beyene and Mantaci \cite[p.\ 5]{fr} to construct a combinatorial bijection between the two sets.

Let us consider the map $f : \mathcal{C}_{n, k} \rightarrow \mathcal{L}_{n, k}$ defined as follows: for each $\pi \in \mathcal{C}_{n, k}$, delete the subword $1X$ and let $\tau$ be the standardized form of the remaining elements. Insert the subword $n(X-1)$ after the smallest rightmost element to $(X-1)$ in $\tau$ and re- order to get $\sigma = f(\pi) \in \mathcal{L}_{n, k}$.
\begin{example}
Let us illustrate the map $f$ with $n = 5$ and $k = 2$. We have $C_{5, 2} = \{13245, 14235, 15234\}$.
For $\pi = 13245$, then $\tau = 123$ and $\sigma = 15234$ i.e., $f(13245) = 15234$.
Similarly, $f(14235) = 12534$, $f(15234) = 12354$.
\end{example}
\begin{proposition}\label{hhh}
The map $f  : \mathcal{C}_{n, k} \rightarrow \mathcal{L}_{n, k}$ is a bijection.
\end{proposition}
\begin{proof}
We will prove that $f$ is surjective and injective.
\begin{itemize}
\item [(i)]  \textbf{Surjectivity}. Let $g : \mathcal{L}_{n, k} \rightarrow \mathcal{C}_{n, k}$ be defined as follows: for each $\sigma \in \mathcal{L}_{n, k}$, delete the subword $nX$ and let $\tau$ be the standardized form of the remaining. Insert the subword $1(X+1)$ before $\tau$ and re-order to get $\pi \in \mathcal{C}_{n, k}$. It is obvious that $g$ is the inverse map of $f$.
\item[(ii)] \textbf{Injectivity}. Let $\pi_{1}$ and $\pi_{2}$ be two elements in $\mathcal{C}_{n, k}$ such that $f(\pi_{1}) = f(\pi_{2})$. Necessarily, $\tau_{1} = \tau_{2}$ and hence $\pi_{1} = \pi_{2}$.
\end{itemize}
\end{proof}
Let us illustrate the map $f^{-1}$ with $n = 5$ and $k = 2$. We have $\mathcal{L}_{n, k} = \{12354, 12534, 15234\}$.
For $\sigma = 12354$, then $\tau = 123$ and $\pi = 15234$ i.e., $f^{-1}(12354) = 15234$. Similarly, $f^{-1}(12534) = 14235$, $f^{-1}(15234) = 13245$.
\begin{theorem}\label{cvb}
For all integers $n$ and $k$ such that $1 \leq k < n$, the cardinality of the set $\mathcal{K}_{n, k}$ is $kf_{n-1, k}$.
\end{theorem}
\begin{proof}
Let us construct the elements of  $\mathcal{K}_{n, k}$.
Let $\pi \in \mathcal{F}_{n-1, k}$. Inserting $n$ at the end of a run of $\pi$ does not change the number of runs. There are $k$ possibilities of such insertions. This generates $k$ flattened partitions in the set $\mathcal{K}_{n, k}$. Thus $|\mathcal{K}_{n, k}| = kf_{n-1, k}$.
\end{proof}
\begin{theorem}\label{comp}
For all integers $n$ and $k$ such that $1 \leq k < n$, $f_{n, k} = kf_{n-1, k} +(n-2)f_{n-2, k-1}$.
\end{theorem}
\begin{proof}
From Remark \ref{rmk}, we have that $|\mathcal{F}_{n, k}| = |\mathcal{K}_{n, k}| + |\mathcal{L}_{n, k}|$. Using Theorem \ref{cvb}, Proposition \ref{hhh} and Corollary \ref{sss}, we then have that
\begin{equation}\label{comp}
f_{n, k} = kf_{n-1, k} +(n-2)f_{n-2, k-1}.
\end{equation}
\end{proof}
\begin{corollary}\label{uu}
	For all integers $n$ and $k$ such that $1 \leq k < n$, we have \begin{equation*}
	kf_{n, k} = \sum_{m=1}^{n-2}\bigg(\binom{n}{m} - 1\bigg)f_{m, k-1}.
	\end{equation*} 
\end{corollary}
\begin{proof}
Considering Equation \eqref{mum1} and Equation \eqref{comp}, we deduce the result.
\end{proof}

We still do not have a direct combinatorial proof of Corollary \ref{uu}, and it remains an open problem worth investigating.
\begin{theorem}\label{vv}
For all integers $n$ and $k$ such that $1 \leq k \leq n$, the numbers $f_{n+2, k}$ of flattened partitions over $[n+2]$ having $k$ runs satisfy the recurrence relation \begin{equation}\label{oli}
f_{n+2, k} = f_{n+1, k} + \sum_{i = 1}^{n}\binom{n}{i}f_{n+1-i, k-1}.
\end{equation}
\end{theorem} \begin{proof}
To construct a flattened partition $\pi$ over $[n+2]$ for all $n \geq 0$ having $k$ runs, we use the property that the first two terms $1$ and $2$ of $\pi$ are either in the same run or in different runs. We consider these two possibilities to enumerate $f_{n+2, k}$.
\begin{enumerate}
\item If $1$ and $2$ are in the same run, by the construction in Theorem \ref{thm4}, $\pi$ is of the form $\pi = 1\underbrace{\tau}_{n+1}$, where $\tau$ is a subword of length $n+1$ whose starting integer is $2$. Deleting $1$ of $\pi$ and reducing each of the remaining terms by $1$ gives a flattened partition of length $n+1$, with the same number of runs $k$ i.e., $f_{n+1, k}$.
\item  If $1$ and $2$ are in different runs, suppose the first run has $(i+1)$ terms including the first term $1$. Using Theorem \ref{thm4}, the remaining $(i+1) - 1 = i$ terms in the first run can be chosen from the set $\{3, 4, \ldots, n+2\}$ of $(n+2)-2 = n$ terms. This is because the positions of $1$ and $2$ in $\pi$ are already known. There are $\displaystyle {n \choose i}$ ways to do this. The remaining $k-1$ runs have length $(n+2)-(i+1) = n+1-i$. Hence we have $f_{n+1-i, k-1}$ flattened partitions over $[n+1-i]$ with $k-1$ runs. Since the length of the first run including $1$ varies between $2$ and $n+1$, the number of flattened partitions over $[n+2]$ with $1$ and $2$ in different runs is given by $\displaystyle \sum_{i =1}^{n}{n \choose i}f_{n+1-i, k-1}$.
\end{enumerate}
Adding these two cases together gives \begin{equation*}
f_{n+2, k} = f_{n+1, k} + \sum_{i = 1}^{n}{n \choose i}f_{n+1-i, k-1}.
\end{equation*}
\end{proof} 
\subsection{Generating function}
\begin{theorem}\label{jk}
The exponential generating function $F(x, u)$ of the run distribution over flattened partitions has the closed differential form \begin{equation}\label{gg}
\frac{\partial F(x, u) }{\partial u} = x\exp(x(\exp(u) - 1)+ u(1-x)),
\end{equation} with initial condition $\dfrac{\partial F(x, 0)}{\partial u} = x$.
\end{theorem}
\begin{proof}
We have
\begin{equation*}
F(x, u) = \sum_{n \geq 1}\sum_{k\geq 1} f_{n, k} x^{k}\frac{u^{n}}{n!} = \sum_{n \geq 1}f_{n}(x)\frac{u^{n}}{n!} .
\end{equation*} where $f_{n}(x)$ is the polynomial defined by $\displaystyle \sum_{k\geq 1} f_{n, k} x^{k} = f_{n}(x)$.

 From Equation \eqref{oli}, multiplying by $\displaystyle x^k \frac{u^n}{n!}$ and summing over $k$ and $n$ gives
\begin{equation}\label{ooo}
\sum_{n\geq 1}\sum_{k \geq 1}f_{n+2, k}x^k \frac{u^n}{n!} = \sum_{n\geq 1}\sum_{k \geq 1}\bigg(f_{n+1, k} + \sum_{i = 1}^{n}{n \choose i}f_{n+1-i, k-1}\bigg)x^k \frac{u^n}{n!}.
\end{equation}
Equation \eqref{ooo} can be rewritten as $A = B+C$, where \begin{equation*}
A = \sum_{n\geq 1}\sum_{k \geq 1}f_{n+2, k}x^k \frac{u^n}{n!},
\end{equation*}
\begin{equation*}
B = \sum_{n\geq 1}\sum_{k \geq 1}f_{n+1, k}x^k \frac{u^n}{n!},
\end{equation*} and
\begin{equation*}
C = \sum_{n\geq 1}\sum_{k \geq 1}\bigg(\sum_{i = 1}^{n}{n \choose i}f_{n+1-i, k-1}x^k \frac{u^n}{n!}\bigg).
\end{equation*}
We have \begin{equation}A = \frac{\partial^2 F}{\partial u^2},\ B = \frac{\partial F}{\partial u}. \label{op} \end{equation}
Fixing $i$ and summing over $k$ in $C$ gives
\begin{equation}\label{no7}
C = \sum_{n\geq 1}\sum_{i = 1}^{n}{n \choose i}\bigg( \sum_{k \geq 1}f_{n+1-i, k-1}x^k\bigg)\frac{u^{n}}{n!}.
\end{equation}
Equation \eqref{no7} can be re-written as
\begin{equation}\label{no8}
C = x\sum_{n\geq 1}\sum_{i = 1}^{n}{n \choose i}\bigg( \sum_{k \geq 1}f_{n+1-i, k-1}x^{k-1}\bigg)\frac{u^{n}}{n!}.
\end{equation}
Expanding the binomial coefficient $\displaystyle
{n \choose i}$ in Equation \eqref{no8} and simplifying gives
\begin{align*}\label{no9}
C &= x\sum_{n\geq 1}\sum_{i = 1}^{n}\frac{u^i}{i!}\bigg( \sum_{k \geq 1}f_{(n-i)+1, k-1}x^{k-1}\frac{u^{(n-i)}}{(n-i)!}\bigg)\\
& = x\sum_{i = 1}^{\infty}\frac{u^i}{i!}\bigg(\sum_{n\geq i} \sum_{k \geq 1}f_{(n-i)+1, k-1}x^{k-1}\frac{u^{(n-i)}}{(n-i)!}\bigg)\\
& = x\sum_{i = 1}^{\infty}\frac{u^i}{i!}\bigg(\sum_{n \geq i} f_{(n-i)+1}(x)\frac{u^{n-i}}{(n-i)!}\bigg).
\end{align*}
We have
\begin{align*}
C &=  x\sum_{i = 1}^{\infty}\frac{u^i}{i!}\frac{\partial F(x, u)}{\partial u}\\
& = x(\exp(u) - 1)\frac{\partial F(x, u)}{\partial u}.
\end{align*}
Substituting Equation \eqref{op} and $C = x(\exp(u) - 1)\dfrac{\partial F(x, u)}{\partial u}$ into Equation \eqref{ooo} gives \begin{equation}\label{o}
\frac{\partial^2 F(x, u)}{\partial u^2} = \frac{\partial F(x, u)}{\partial u} + x(\exp(u) - 1)\frac{\partial F(x, u)}{\partial u}.
\end{equation}
Let $V = \dfrac{\partial F(x, u)}{\partial u}$. Then \begin{equation*}
\frac{\partial V}{\partial u} = \frac{\partial^2 F(x, u)}{\partial u^2}.
\end{equation*}  Substituting $V$ into Equation \eqref{o} gives
\begin{equation}\label{no11}
\frac{\partial V}{\partial u} = V\big(x(\exp(u) - 1)+1\big)
\end{equation} with initial condition $V(x, 0) = x$. Solving for $V$ in Equation \eqref{no11} gives \begin{equation*}
V =  \frac{\partial F(x, u)}{\partial u} =  x \exp(-x)\exp({(x(\exp(u) - u)+u)}).
\end{equation*}
\end{proof}
\section {Flattened partitions with the first $s$ terms in different runs}\label{C}
We recall that $f_{n, k}^{(s)}$ and $F^{[s]}(x, u)$ are the number of flattened partitions over $[n]$ whose first $s$ integers belong to different runs and the exponential generating function for the numbers $f_{n, k}^{(s)}$ respectively.
\begin{theorem}
For all integers $s, k$ and $n$ such that $1 \leq s \leq k < n$, the numbers $f_{n+s, k}^{(s)}$ satisfy the relation \begin{equation}\label{ppp}
f_{n+s, k}^{(s)} = \sum_{i_{1}, i_{2}, \ldots , i_{s} \geq 1}\binom{n}{i_{1}}\binom{n-i_{1}}{i_{2}}\cdots \binom{n-\sum_{j=1}^{s-1}i_{j}}{i_{s}} f_{n-\sum_{j=1}^{s}i_{j}, k-s}.
\end{equation}
\end{theorem}
\begin{proof}
Let $\pi$ be a flattened partition over $[n+s]$ having $k$ runs. Let $i_{1}+1, i_{2}+1, i_{3}+1, \ldots , i_{s}+1$ be the lengths of the $s$ first runs whose starting points are $1, 2, \ldots, s$ respectively. Since the first run, including $1$ has length $i_{1}+1$, we have $i_{1}$ terms to arrange out of the $(n+s)-s = n$ terms. There are ${n \choose i_{1}}$ ways. For the second run, it remains to arrange $i_{2}$ terms out of $n-i_{1}$ terms. There are $\displaystyle {n-i_{1} \choose i_{2}}$ ways. Repeating the same process up to the $s^{th}$ run inductively gives $\displaystyle {n-i_{1} - i_{2}- \cdots - i_{s-1} \choose i_{s}} = {n-\sum_{j=1}^{s-1}i_{j} \choose i_{s}}$ possibilities.
The remaining $k-s$ runs have length $(n+s)-((i_{1}+1) + (i_{2}+1) + \cdots (i_{s}+1)) = \displaystyle n- \sum_{j=1}^{s}i_{j}$. So we have $\displaystyle {n \choose i_{1}}{n-i_{1} \choose i_{2}}\cdots {n-\sum_{j=1}^{s-1}i_{j} \choose i_{s}} f_{n-\sum_{j=1}^{s}i_{j}, k-s}$. Summing over all possibilities of $i_{1}, i_{2}, \ldots , i_{s} \geq 1$ gives the result.
\end{proof}
\begin{theorem}
 The exponential generating function $F^{[s]}(x, u)$ for the numbers $f_{n, k}^{(s)}$ has the closed differential form \begin{equation}\label{ddw}
\frac{\partial^s F^{[s]}(x, u)}{\partial u^s} = (\displaystyle x(\exp(u) - 1))^{s} F(x, u).
\end{equation}
\end{theorem}
\begin{proof}
We have 
\begin{equation*}	F^{[s]}(x, u) = \sum_{n=1}^{\infty}\sum_{k\geq 1} f_{n+s, k}^{(s)} x^{k}\frac{u^{n}}{n!}.
\end{equation*} 
 From Equation \eqref{ppp}, multiplying by $\displaystyle x^k \frac{u^n}{n!}$ and summing over $k$ and $n$ gives
\begin{equation}\label{nv}
\sum_{n\geq 1}\sum_{k \geq s}f_{n+s, k}^{(s)}x^k \frac{u^n}{n!} = \sum_{n\geq 1}\sum_{k \geq s}\bigg(\sum_{i_{1}, i_{2}, \ldots , i_{s} \geq 1}{n \choose i_{1}}\cdots {n-\sum_{j=1}^{s-1}i_{j} \choose i_{s}} f_{n-\sum_{j=1}^{s}i_{j}, k-s}\bigg)x^k \frac{u^n}{n!}
\end{equation}
Let \begin{equation*}
L = \sum_{n\geq 1}\sum_{k \geq s}f_{n+s, k}^{(s)}x^k \frac{u^n}{n!},
\end{equation*} and \begin{equation*}
M = \sum_{n\geq 1}\sum_{k \geq s}\bigg(\sum_{i_{1}, i_{2}, \ldots , i_{s} \geq 1}{n \choose i_{1}}{n-i_{1} \choose i_{2}}\cdots {n-\sum_{j=1}^{s-1}i_{j} \choose i_{s}} f_{n-\sum_{j=1}^{s}i_{j}, k-s}\bigg)x^k \frac{u^n}{n!}.
\end{equation*}
We have \begin{equation}\label{hj}L = \dfrac{\partial^{s}F^{[s]}(x, u)}{\partial u^{s}}. \end{equation} Fixing $i$ and summing over $k$ gives
\begin{align*}
M &= \sum_{n \geq 1}\sum_{i_{1}, i_{2}, \ldots , i_{s} \geq 1}\bigg(\sum_{k} x^s {n \choose i_{1}}{n-i_{1} \choose i_{2}}\cdots {n-\sum_{j=1}^{s-1}i_{j} \choose i_{s}}f_{n-\sum_{j=1}^{s}i_{j}, k-s} x^{k-s} \bigg)\frac{u^n}{n!}\\
& = x^{s}\sum_{i_{1}, i_{2}, \ldots , i_{s} = 1}^{\infty}{n \choose i_{1}}{n-i_{1} \choose i_{2}}\cdots {n-\sum_{j =1}^{s-1}i_{j} \choose i_{s}}\bigg(\sum_{n \geq i_{1}, i_{2}, \ldots , i_{s}}\sum_{k \geq s}f_{n-\sum_{j =1}^{s}i_{j}, k-s}x^{k-s}\frac{u^n}{n!}\bigg)\\
& = x^s \sum_{i_{1}, i_{2}, \ldots , i_{s} = 1}^{\infty}\frac{n!}{(n-i_{1})! i_{1}!}\cdots \frac{(n-\sum_{j=1}^{s-1}i_{j})!}{ (n-\sum_{j=1}^{s}i_{j})!i_{s}!} \bigg(\sum_{n \geq i_{1}, i_{2}, \ldots , i_{s}}\sum_{k \geq s}f_{n-\sum_{j=1}^{s}i_{j}, k-s}x^{k-s}\frac{u^n}{n!}\bigg)\\
& = x^s \sum_{i_{1}, i_{2}, \ldots , i_{s} = 1}^{\infty}\frac{u^{i_{1}}}{i_{1}!}\frac{u^{i_{2}}}{i_{2}!}\cdots \frac{u^{i_{s}}}{i_{s}!}\bigg(\sum_{n\geq i_{1}, i_{2}, \ldots , i_{s}}\sum_{k \geq s}f_{n-\sum_{j =1}^{s}i_{j}, k-s}x^{k-s}\frac{u^{n-\sum_{j=1}^{s}i_{j}}}{(n-\sum_{j=1}^{s}i_{j})!}\bigg)\\
& = (\displaystyle x(\exp(u) - 1))^{\normalsize s} F(x, u).
\end{align*}
Substituting Equation \eqref{hj} and $M = (\displaystyle x(\exp(u) - 1))^{\normalsize s} F(x, u)$ into Equation \eqref{nv} gives the result.
\end{proof}
\section{Bijection between flattened partitions over $[n+1]$ and partitions of $[n]$}\label{A}
Let $P \in \mathcal{P}_{n}$ be a partition of $[n]$, written as $P = B_{1}|B_{2}|\cdots|B_{k}$ where the elements in each block $B_{i}$ are written in increasing order. We will write $P$ as $P'$ in such a way that in each block, the smallest element appears at the end but still maintaining the remaining elements. We construct a word $P_{+}$ from $P'$ by deleting the marks $``| "$ enclosing the different blocks of $P'$ and then increasing all the terms by $1$.

Let $f : \mathcal{P}_{n} \rightarrow \mathcal{F}_{n+1}$ be a map which associates a partition $P$ of $[n]$ with a flattened partition defined by $f(P) = \sigma = 1P_{+}$. The map $f$ is well defined since if $P_{1}$ and $P_{2}$ are two set partitions over $[n]$ and assume that $P_{1} = P_{2}$, then $f(P_{1})  = 1P_{1+} = 1P_{2+} = f(P_{2})$.
\begin{proposition}
The map $f : \mathcal{P}_{n} \rightarrow \mathcal{F}_{n+1}$ is a bijection.
\end{proposition}
\begin{proof}
Since both $\mathcal{P}_{n}$ and $\mathcal{F}_{n+1}$ are finite with the same size $B_{n}$, it suffices to prove that $f$ is injective. Let $P_{1}$ and $P_{2}$ be two partitions of $[n]$. Assume that $f(P_{1}) = f(P_{2})$. Then by definition, $1P_{1+} = 1P_{2+}$. Since the map $f$ is well defined, the strings $P_{1+}$  and $P_{2+}$ are equal after deleting $1$ from the front. Thus $P_{1} = P_{2}$.
\end{proof}
\begin{example}\label{far}
Consider a partition $P = 12|3|45$ of $[5]$. Re-ordering $P$ such that in each block, the smallest entry appears at the end of the partition gives $P' = 21|3|54$. Then $P_{+} = 32465$. And thus $\sigma = 1P_{+} = 132465$.
\end{example} Let $\sigma \in \mathcal{F}_{n+1}$ having $k$ runs. Let us define $P= f^{-1}(\sigma)$ as follows: Insert the mark $``|"$ at the end of each right to left minimum of $\sigma$. Delete element $1$.  Re-order to get the blocks of the partition $P$.
\begin{example}\label{nol1}
Consider a flattened partition $\sigma = 132465$ over $[6]$. Inserting the mark $``|"$ at the end of each right to left minimum $1|32|4|65|$. Deleting $1$ and re-ordering the remaining gives the partition $P = 12|3|45$.
\end{example}
We now give a property of our bijection. This bijection preserves the number of blocks of size greater than $1$ in a partition and the number of runs of its corresponding flattened partition.
\begin{theorem}
For any integer $n \geq 0$, if $P$ is a partition over $[n]$ and $\sigma$ its corresponding flattened partition over $[n+1]$, then the following assertions are equivalent:
\begin{enumerate}
\item [(i)] the number of blocks of size greater than $1$ of the partition $P$ is $k-1$,
\item [(ii)] the number of runs of the flattened partition $\sigma$ is equal to $k$.
\end{enumerate}
\end{theorem}
\begin{proof}
For $n = 0$, $1$ is the only flattened partition corresponding to the empty set.
Consider a partition $P$ of $[n]$ having $(k-1)$ blocks of size greater than $1$. Using the construction of the map $f$, we write each of the blocks in such a way that the smallest elements of each block appear at the end. Thus the smallest elements become the starting points for a run in a flattened partition $\sigma = f(P)$ if the block has at least two elements as well as integer $1$. The element of singleton blocks becomes a right to left minimum in $\sigma$ and is not a starting point of a run. So $\sigma$ has $k$ runs.
	
Conversely, consider a flattened partition $\sigma$ having $k$ runs. Placing a mark $``|"$ after the right to left minima will form a partition of $[n+1]$. Since a starting point of a run is preceded by a greater element, except the integer $1$, the blocks with two or more elements consist of those containing a starting point. Deleting the integer $1$ and re-ordering the remaining blocks will give $(k-1)$ blocks of size greater than $1$.
\end{proof}

In Example \ref{far}, the partition $P$ has $2$ blocks of size greater than $1$ and the corresponding flattened partition, $\sigma$ has $3$ runs.

\section{Acknowledgment} 
We would like to thank the editor and the anonymous referee for their comments and useful suggestions. This work was financially supported by the Swedish Sida bilateral program with Makerere University, 2015-2020, project 316 ``Capacity building in Mathematics and its applications". The first author would like to thank in a special way Prof.\ Paul Vaderlind (Stockholm University), Prof.\ Roberto Mantaci (University Paris Diderot 7) and the D\'epartement de Math\'ematiques et Informatique, Universit\'e d' Antananarivo for hosting her during research visits.

\bigskip
\hrule
\bigskip
\noindent 2010 {\it Mathematics Subject Classification}:
Primary 05A05; Secondary 05A10, 05A15, 05A18. 
\noindent \emph{Keywords:} flattened partition, generating function, recurrence relation, run, set partition.
\bigskip
\hrule
\bigskip
\noindent (Concerned with sequence
\seqnum{A000295} and \seqnum{A124324}.)

\end{document}